\documentclass[a4paper, reqno, 12pt]{amsart}
\usepackage[utf8]{inputenc}
\usepackage[T1, T2A]{fontenc}
\usepackage[english]{babel}
\usepackage{amsmath, amssymb, amsfonts, amsthm, amscd, mathrsfs}
\usepackage{enumitem}
\usepackage[hidelinks]{hyperref}
\usepackage[usenames]{color}
\usepackage{longtable}
\setlist[enumerate]{label = (\alph*), ref=(\text{\alph*)}}
\setlist[itemize]{nolistsep}

\usepackage{epsfig}
\usepackage{graphicx}
\usepackage[cmtip,arrow]{xy}
\usepackage{bm}
\usepackage{tikz}
\usepackage{makecell}
\usepackage{multirow}
\usepackage{diagbox}
\usetikzlibrary{calc}
\tikzstyle{precornfill} = [fill=black!10, draw = black]
\tikzstyle{cornfill} = [fill=black!30, draw = black]

\renewcommand{\phi}{\varphi}

\newcommand{\KK}{\mathbb{K}}
\renewcommand{\AA}{\mathbb{A}}

\newcommand{\GG}{\mathbb{G}}
\newcommand{\PP}{\mathbb{P}}
\newcommand{\mm}{\mathfrak{m}}
\newcommand{\NN}{\mathbb{Z}_{>0}}

\DeclareMathOperator{\GL}{GL}

\DeclareMathOperator{\Aut}{Aut}

\DeclareMathOperator{\Lie}{Lie}
\DeclareMathOperator{\Mat}{Mat}
\DeclareMathOperator{\Soc}{Soc}

\DeclareMathOperator{\length}{length}

\theoremstyle{plain}
\newtheorem{lemma}{Lemma}
\newtheorem{proposition}{Proposition}
\newtheorem{theorem}{Theorem}

\theoremstyle{definition}
\newtheorem{definition}{Definition}

\theoremstyle{remark}

\begin{document}

\title[Projective Hypersurfaces of High Degree with an Additive Action
]{Projective Hypersurfaces of High Degree \\ Admitting an Induced
Additive Action
}

\author{Ivan Beldiev}
\address{HSE University, Faculty of Computer Science, Pokrovsky Boulvard 11, Moscow, 109028 Russia}
\email{isbeldiev@hse.ru, ivbeldiev@gmail.com}

\thanks{The article was prepared within the framework of the project “International academic cooperation” HSE University.}

\subjclass[2010]{Primary 14L30, 13E10; \ Secondary 14J70, 14J17}

\keywords{Projective space, hypersurface, commutative unipotent group, local algebra, Gorenstein algebra}

\begin{abstract}
We study induced additive actions on projective hypersurfaces, i.e. effective regular actions of the algebraic group $\mathbb G_a^m$ with an open orbit that can be extended to a regular action on the ambient projective space. It is known that the degree of a hypersurface $X\subseteq\PP^n$ admitting an induced additive action cannot be greater than $n$ and there is a unique such hypersurface of degree $n$. We give a complete classification of hypersurfaces $X\subseteq \PP^n$ admitting an induced additive action of degrees from $n-1$ to $n-3$.
\end{abstract}

\maketitle

\section{Introduction}
\label{intro}

In the paper, we assume that the ground field $\KK$ is algebraically closed of characteristic zero. By an algebraic variety, we mean an algebraic variety over $\KK$. We denote by $\GG_a^m$ the algebraic group $(\KK^m, +)$.

An \emph{additive action} on an algebraic variety $X$ is an effective regular action of the group~$\GG_a^m$ on $X$ with an open orbit. In this paper, we consider only the case when $X\subseteq \PP^n$ is a projective hypersurface and an additive action on $X$ is \emph{induced}, i.e. can be extended to a regular action of $\GG_a^m$ on the ambient projective space $\PP^n$. It is clear that $n = m + 1$ for dimension reasons.

Many results on additive actions were obtained during the last decades. For example, all projective toric hypersurfaces admitting an additive action are classified in \cite{Sha-2}. In~\cite{Liu}, the author obtains a classification of additive actions on hyperquadrics of corank~$2$ whose singularities are not fixed by these actions. Other recent results can be found in \cite{AP, AR, AS, BGT, Sha}.

There is a bijection between additive actions on~$\PP^n$ and local finite-dimensional commutative associative unital algebras of dimension~$n+1$ established by Hassett and Tschinkel in \cite{HT}. This bijection is called the Hassett-Tschinkel correspondence. In \cite{AZa}, a generalized version of this correspondence is suggested. It turns out that there is, up to equivalences, a bijection between the following objects:

\begin{enumerate}
    \item induced additive actions on projective hypersurfaces in $\PP^n$ that are not a hyperplane;
    \item pairs $(A, U)$, where $A$ is a local commutative associative unital algebra over $\KK$ of dimension $n+1$ with the maximal ideal $\mm$ and $U\subseteq \mm$ is a hyperplane generating the algebra $A$. Such pairs $(A,U)$ are called \emph{$H$-pairs}.
\end{enumerate}

It turns out that the existence of an induced additive action on a projective hypersurface is a strong condition. For example, a smooth hypersurface $X$ admits an induced additive action if and only if $X$ is a non-degenerate quadric. It is also proved in~\cite[Corollary~5.2]{AS} that the degree $d$ of a hypersurface $X\subseteq \PP^n$ admitting an induced additive action cannot be greater than $n$. 

It is natural to try to describe projective hypersurfaces of other high degrees admitting an induced additive action. It suffices to consider only \emph{non-degenerate} hypersurfaces i.e. hypersurfaces which are not isomorphic to a projective cone over a hypersurface in a smaller projective space. Indeed, it is shown in \cite[Proposition 2.20 and Corollary 2.23]{AZa} that an additive action on a degenerate hypersurface can be effectively reduced to an additive action on a non-degenerate hypersurface in a smaller projective space having the same equation (and hence the same degree). It is known (see \cite[Theorem 2.30]{AZa}) that non-degenerate hypersurfaces of degree $d$ in $\PP^n$ correspond to $H$-pairs $(A,U)$, where $A$ is a Gorenstein local algebra of dimension $n+1$ with the maximal ideal $\mm$ and $U$ is a hyperplane in $\mm$ complementary to $\mm^d$. It is conjectured in \cite{AZa} that the hypersurface and the additive action on it do not depend on the choice of $U$.

The particular case $d = n$ is studied in \cite{ABeZa}. It turns out that for each $n\geq 2$ there exists a unique hypersurface in $\PP^n$ of degree $n$ with an induced additive action. The corresponding local algebra is $\KK[x]/(x^{n+1})$.

In this paper, we give a complete classification of non-degenerate hypersurfaces in $\PP^{n}$ of degrees $n-1$, $n-2$ and $n-3$ admitting an induced additive action. The hypersurfaces are described in terms of the corresponding Gorenstein algebras. For each of them, we verify that the corresponding non-degenerate hypersurface does not depend on the choice of the subspace $U$. Our work is largely based on the classification of certain classes of Gorenstein local algebras obtained in \cite{Casn} and especially in \cite{ElVa}.

The number of isomorphism classes of non-degenerate hypersurfaces in $\PP^n$ of degrees $d$ from $n-1$ to $n-3$ admitting an induced additive action is given by the following table.

\vspace{0.2cm}

\begin{center}
\begin{tabular}{ |c|c|c|c|c|c|c|c|c|  }
\hline
  \diagbox{$d$}{$n$}& $3$ & $4$ & $5$ & $6$ & $7$ & $8$ & $9$ & $\geq 10$\\
\hline
$n-1$ & \multicolumn{8}{|c|}{1}\\
\hline
$n - 2$ & -- & $1$ & $2$ & \multicolumn{5}{|c|}{3}\\
\hline
$n - 3$ & -- & -- & $1$ & $3$ & $6$ & $5$ & $\infty$ & $6$ \\
\hline
\end{tabular}
\end{center}

\vspace{0.05cm}

\begin{center}
Table 1: The number of non-degenerate hypersurfaces in $\PP^n$ of degree $d$\\
admitting an induced additive action
\end{center}

\vspace{0.3cm}

It turns out that there is a unique non-degenerate hypersurface of degree $n-1$ in~$\PP^n$ admitting an induced additive action for any $n \geqslant 3$. This hypersurface is normal if and only if $n = 3$ or $n = 4$. The equations of this hypersurface, for example, for $n = 4$ and $n = 6$, are
$$z_0^2z_3 - z_0z_1z_2 - \frac{1}{2}z_0z_4^2 + \frac{1}{3}z_1^3 = 0 \text{ and}$$
$$z_0^4z_5-z_0^3z_1z_4 - z_0^3z_2z_3 - \frac{1}{2}z_0^2z_6^2 + z_0^2z_1^2z_3 + z_0^2z_1z_2^2 - z_0z_1^3z_2 + \frac{1}{5}z_1^5 = 0.$$


For $d = n - 3$ we have $6$ isomorphism classes of such hypersurfaces if $n\geqslant 10$. For $5\leqslant n \leqslant 9$, the number of such hypersurfaces is also computed and is given in Theorem \ref{codim_three}. The most interesting case is $n = 9$, when an infinite family of non-isomorphic hypersurfaces arises. The corresponding Gorenstein local algebras are
$$\KK[x,y]/(y^2 - x^2y - cx^4, x^3y), \quad c\in\KK^*.$$
These algebras (and hence the corresponding hypersurfaces) are pairwise not isomorphic for different $c\in\KK^*$. All of the corresponding hypersurfaces are not normal.

\section{Preliminaries}

We start with several definitions and results on local finite-dimensional algebra. Recall that an algebra is called \emph{local} if it has a unique maximal ideal $\mm$. All algebras in this paper are assumed to be commutative, associative, and unital.

\begin{lemma} \cite[Lemma 1.2]{AZa}
    A finite-dimensional algebra $A$ is local if and only if $A$ is the direct sum of its subspaces $\KK \oplus \mm$, where $\mm$ is the ideal consisting of all nilpotent elements of $A$.
\end{lemma}

\begin{definition}
    The \emph{socle} of a local algebra $A$ with the maximal ideal $\mm$ is the ideal $$\Soc A = \{a\in A\mid a \mm = 0\}.$$ A local finite-dimensional algebra $A$ is called \emph{Gorenstein} if $\dim \Soc A = 1$.
\end{definition}

If $d$ is the maximal number such that $\mm^d\ne 0$, then $\mm^d\subseteq \Soc A$. However, this inclusion can be strict. So, $A$ is Gorenstein if and only if $\dim \mm^d = 1$ and $\Soc A = \mm^d$.

Let $A$ be a local finite dimensional algebra and $\mm$ its maximal ideal. Consider the following sequence of ideals in $A$:
$$A\supset \mm\supset \mm^2 \supset \ldots \supset \mm^d \supset \mm^{d+1} = 0.$$
The number $d+1$ is called the \emph{length} of the algebra $A$. Denote $r_i = \dim \mm^i - \dim\mm^{i+1}$; in particular, $r_0 = 1$.
\begin{definition}
    The sequence $r_0, r_1, \ldots, r_{d}$ is called the \emph{Hilbert-Samuel sequence} of the algebra $A$.
\end{definition}

Next, we give a formal definition of equivalence of induced additive actions on projective hypersurfaces.

\begin{definition}
    Two induced additive actions $\alpha_i\colon \GG_a^m \times X_i\to X_i$, $X_i\subseteq \PP^n$, $i = 1, 2$, are called \emph{equivalent} if there exists an automorphism of algebraic groups $\phi\colon \GG_a^m \to \GG_a^m$ and an automorphism $\psi \colon \PP^n \to \PP^n$ such that $\psi(X_1) = X_2$ and $\psi\circ\alpha_1=\alpha_2\circ(\phi\times\psi).$
\end{definition}

Let us give the definition of an $H$-pair.

\begin{definition}
\label{pdpd}
An \emph{$H$-pair} is a pair $(A,U)$, where $A$ is a local finite-dimensional algebra with the maximal ideal $\mm$ and $U\subseteq \mm$ is a hyperplane generating $A$ as a unital algebra.
\end{definition}

One can define equivalence of $H$-pairs as follows.

\begin{definition}
    Two pairs $(A_1, U_1)$ and $(A_2, U_2)$ are called \emph{equivalent} if there exists an isomorphism of algebras $\phi\colon A_1 \to A_2$ such that $\phi(U_1) = U_2$.
\end{definition}

Now, we give the precise statement of the generalized version of the Hassett-Tschinkel correspondence.

\begin{theorem}\cite[Theorem 2.6]{AZa}
\label{prth}
    Suppose $n\in \NN$. There is a one-to-one correspondence between the following objects:
    \begin{enumerate}
        \item induced additive actions on hypersurfaces in $\PP^n$ that are not a hyperplane;
        \item pairs $(A,U)$, where $A$ is a local commutative associative unital algebra of dimension~$n+1$ with the maximal ideal $\mm$ and $U\subseteq \mm$ is a hyperplane generating the algebra~$A$.
    \end{enumerate}
    This correspondence is considered up to equivalences from Definitions 3 and 4.
\end{theorem}

The construction of this correspondence is done as follows. For an $H$-pair $(A, U)$, denote by $p\colon A \setminus \{0\} \to \PP(A)\cong \PP^n$ the canonical projection. Then we define
$$X = p(\overline{\KK^{\times}\exp U}),$$
i.e., it is the projectivization of the Zariski closure of the subset $\KK^{\times}\exp U\subseteq A\setminus\{0\}$. Since $A$ is commutative, the algebraic group $\exp U$ can be identified with $\GG_a^{n-1}$; hence the multiplication by elements of $\exp U$ defines an action of $\GG_a^{n-1}$ on $\PP(A)$. It is easy to see that $X$ is preserved under this multiplication, so this defines an induced additive action of $\GG_a^{n-1}$ on $X \subseteq \PP(A) \cong \PP^n$.

Conversely, an induced additive action of $\GG_a^{n-1}$ on a hypersurface $X\subseteq \PP^n = \PP(V)$, where $\dim V = n+1$, can be lifted to a linear action of $\GG_a^{n-1}$ on $V$, which gives us a faithful representation $\rho \colon \GG_a^{n-1} \to \GL_{n+1}(\KK)$. Let $U$ be the vector space $d\rho(\mathfrak g_a^{n-1})$ and define $A$ as the unital subalgebra of $\Mat_{n+1}(\KK)$ generated by $U$; here
$$
d\rho\colon \mathfrak g_a^{n-1} = \Lie(\GG_a^{n-1})\to \Mat_{n+1}(\KK)
$$
is the differential of the map $\rho$. It can be checked that $(A, U)$ is an $H$-pair. One can find the details in \cite[Theorem~1.38]{AZa}.

Given an $H$-pair $(A, U)$, the equation of the corresponding hypersurface is computed as follows (for more details, see \cite[Chapter 2.2]{AZa}). Denote by $\pi$ the canonical projection $\pi\colon U \to U/\mm$. Let $d$ be the greatest positive integer such that $\mm^d \nsubseteq U$. The corresponding projective hypersurface is given by the homogeneous equation
\begin{equation}
\label{hyp_eq}
z_0^d\pi\left(\ln\Bigl(1 + \frac{z}{z_0}\Bigr)\right) = 0
\end{equation}
for $z_0 + z\in A = \KK \oplus \mm$, $z_0\in\KK$, $z\in \mm$. This hypersurface is irreducible and has degree~$d$.

It follows that the degree of a hypersurface $X\subseteq\PP^n$ admitting an induced additive action is at most $n$ since the length of the corresponding algebra is at most $n$.

We are particularly interested in so-called non-degenerate hypersurfaces. It turns out that additive actions on such hypersurfaces correspond to Gorenstein algebras.

\begin{definition}
   The hypersurface $X$ given by the equation $f(z_0, z_1, \ldots, z_n) = 0$, where $f$ is a homogeneous polynomial, is called \emph{non-degenerate} if one of the following equivalent conditions holds:
    \begin{enumerate}
    \item there exists no linear transform of variables such that the number of variables in $f$ after this transform becomes less than $n+1$;
    \item the hypersurface $X$ is not a projective cone over a hypersurface $Z\subseteq \PP^k$ in a projective subspace $\PP^k \subseteq \PP^n$ for some $k < n$.
    \end{enumerate}
\end{definition}

\begin{theorem} \cite[Theorem 2.30]{AZa}
\label{tgor}
Induced additive actions on non-degenerate hypersurfaces of degree $d$ in $\PP^n$ are in one-to-one correspondence with $H$-pairs $(A, U)$, where $A$ is a Gorenstein local algebra of dimension~$n+1$ with the socle $\mm^d$ and $\mm = U \oplus \mm^d$. 
\end{theorem}

In \cite{AP,AZa}, the authors consider the procedure of reduction of an induced additive action. Namely, consider an $H$-pair $(A,U)$ corresponding to an induced additive action on a hypersurface $X$. Let $J \subseteq A$ be an ideal of dimension $n - k$ contained in $U$.

\begin{proposition} \cite[Proposition 2.20 and Corollary 2.23]{AZa}
\label{prop_reduct}
The pair $(A/J, U/J)$ corresponds to an induced additive action on a projective hypersurface $Z\subseteq \PP^{k}$, and $X$ is the projective cone over $Z$, i.e., for some choice of coordinates in $\PP^{n}$ and $\PP^{k}$, the equations of the hypersurfaces~$X$ and $Z$ are the same. Moreover, if $J$ is the maximal (with respect to inclusion) ideal of~$A$ contained in~$U$, then $Z$ is a non-degenerate hypersurface in $\PP^k$.
\end{proposition}

It follows from Proposition \ref{prop_reduct} that the problem of classifying hypersurfaces of high degrees admitting an induced additive actions is reduced to classifying all such non-degenerate hypersurfaces. Indeed, if there is an additive action on a degenerate hypersurface, there is also an additive action on a non-degenerate hypersurface with the same equation (and hence with the same degree) but lying in a smaller projective space. Note that here we speak about classification of hypersurfaces themselves and not additive actions on them. The problem of classifying all additive actions on a degenerate hypersurface is a more delicate question, see \cite{Bel}. However, any non-degenerate hypersurface admits, up to equivalence, at most one induced additive action (see \cite[Theorem 2.32]{AZa}), so classifying all induced additive actions on non-degenerate hypersurfaces is equivalent to classifying simply such hypersurfaces.

\section{Results on Gorenstein algebras}

In this section, we list several results on local finite-dimensional and in particular Gorenstein algebras which we are going to use. First, we need the following technical lemma.

\begin{lemma}\label{degrees} \cite[Lemma 2.13]{AZa}
    Suppose that $\mm$ is the maximal ideal of a local commutative associative algebra $A$. Then for any $k\in\NN$ the space $\mm^k/\mm^{k+1}$ is linearly spanned by the elements $z^k$, $z\in\mm$.
\end{lemma}

The following lemma describes certain restrictions on the Hilbert-Samuel sequence of a local finite-dimensional algebra.

\begin{lemma}\label{alm_str}
    Let $r_0 = 1, r_1, r_2, \ldots, r_d$ be the Hilbert-Samuel sequence of a local finite-dimensional algebra $A$.
    \begin{enumerate}
        \item If $r_k = 1$ for some $k$, then $r_k = r_{k+1} = \ldots = r_d = 1$.
        \item If $r_2 = 2$, then $r_2 = \ldots = r_s = 2$ and $r_{s+1} = \ldots = r_d = 1$ for some $s\leq d$.
    \end{enumerate}
\end{lemma}

\begin{proof}

Both statements follow from the classical theorem of Macaulay, see \cite{Mac}. However, they can be proved in a more elementary way as follows.
    \begin{enumerate}
        \item By Lemma \ref{degrees}, there is $x\in\mm$ such that $\mm^k/\mm^{k+1} = \langle x^k \rangle$. Let us show that $\mm^{k+1}/\mm^{k+2}$ is generated by $x^{k+1}$. We need to prove that for any $y_1, y_2, \ldots, y_{k+1}\in\mm$ there exists $a\in\KK$ such that $$y_1y_2\ldots y_{k+1} - ax^{k+1}\in \mm^{k+2}.$$
        First, since $\mm^{k}/\mm^{k+1} = \langle x^k\rangle$, we have $y_1y_2\ldots y_k - bx^k \in \mm^{k+1}$ for some $b\in\KK$. This implies $y_1y_2\ldots y_{k+1}- bx^ky_{k+1} \in \mm^{k+2}$. Also, $x^{k-1}y_{k+1} - cx^k \in\mm^{k+1}$ for some $c\in\KK$, so $x^ky_{k+1} - cx^{k+1}\in\mm^{k+2}$. It follows that $y_1y_2\ldots y_ky_{k+1} + cx^{k+1} \in\mm^{k+2}$, which finishes the proof.
        \item By Lemma \ref{degrees}, there exist $x, y\in\mm$ such that $\mm^2/\mm^3 = \langle x^2, y^2\rangle$. First, let us show by induction on $k$ that $\mm^k/\mm^{k+1}$ is spanned by monomials in $x,y$ of degree $k$ for any $k\geq 2$. The base case $k = 2$ follows from the assumption $\dim \mm^2/\mm^3 = 2$. Let us prove the induction step from $k$ to $k + 1$. We need to show that for any $z_1, z_2, \ldots, z_{k+1}$ there exists a homogeneous polynomial $f(x,y)$ of degree $k + 1$ such that $z_1z_2\ldots z_kz_{k+1} - f(x,y) \in\mm^{k+2}$. By induction hypothesis, there exists a homogeneous polynomial $g(x,y)$ of degree $k$ such that $z_1z_2\ldots z_k - g(x,y)\in \mm^{k+1}$. This implies $z_1z_2\ldots z_kz_{k+1} - z_{k+1}g(x,y) \in \mm^{k+2}$. Since the constant term of $g$ is zero (otherwise, the element $z_1z_2\ldots z_k - g(x,y)$ would be invertible) and $z_{k+1}x$, $z_{k+1}y$ can be expressed as polynomials in $x$, $y$, the step is proved.
        
        Next, the three elements $x^2, y^2, xy$ are linearly dependent in $\mm^2/\mm^3$, so there is a quadratic homogeneous polynomial $h$ in two variables such that $h(x,y)\in\mm^3$. Consider the following two cases.

        \textit{Case 1:} the polynomial $h$ is the product of two non-proportional linear factors $\widetilde{x}, \widetilde{y}$. We can replace the variables $x, y$ with $\widetilde{x}, \widetilde{y}$ and for brevity denote $\widetilde{x}, \widetilde{y}$ simply by $x,y$. After that, we have $xy\in\mm^3$. It is easy to see now that $\mm^k/\mm^{k+1} = \langle x^k, y^k \rangle$. Indeed, we already proved that $\mm^k/\mm^{k+1}$ is spanned by $x^k, x^{k-1}y, \ldots, y^k$. Since $xy\in\mm^3$, we have $x^iy^{k-i}\in\mm^{k+1}$ if both $i$ and $k-i$ are positive, so we are done.

        \textit{Case 2:} the polynomial $h$ is a perfect square $\widetilde{y}^2$. Choose any $\widetilde{x}$ that is not proportional to $\widetilde{y}^2$ in $\mm/\mm^2$. Writing simply $x,y$ instead of $\widetilde{x}, \widetilde{y}$, we have $y^2 \in \mm^3$ and $\mm^2/\mm^3 = \langle x^2, xy\rangle$. It follows that $\mm^k/\mm^{k+1} = \langle x^k, x^{k-1}y \rangle$ since $x^{k-i}y^i \in \mm^{k+1}$ for any $i \geq 2$.

        We proved that $r_i \leq 2$ for any $i \geq 2$. Denote by $s$ the minimal index such that $r_{s+1} \ne 2$. Then either $r_{s+1} = 0$ and then $r_{i} = 0$ for all $i\geq s + 1$ or $r_{s+1} = 1$ and then $r_{s+1} = r_{s+2} = \ldots = r_d = 1$ by the first part of the lemma.
    \end{enumerate}
\end{proof}

Next, we need several results on classification of Gorenstein local algebras. In \cite{ElVa}, the case of so-called almost stretched Gorenstein algebras is considered.

\begin{definition}
    A local finite-dimensional algebra with the maximal ideal $\mm$ is called \emph{almost stretched} if the minimal number of generators of $\mm^2$ is $2$.
\end{definition}

By Lemma \ref{alm_str}, the Hilbert-Samuel sequence of an almost stretched local algebra $A$ is $(1, h, 2, \ldots, 2, 1, 1, \ldots, 1)$ meaning that $r_1 = h\geqslant 2$, $r_2 = \ldots = r_t = 2$, $r_{t+1} = \ldots = r_{s} = 1$. In this case, $A$ is called almost-stretched of type $(s,t)$.

In \cite{ElVa}, the authors obtain a classification of Gorenstein almost stretched algebras of type $(s,t)$ for so-called regular pairs $(s,t)$ such that $s \geq 2t-1$. The definition of a regular pair $(s,t)$ introduced in \cite{ElVa} is the following.

\begin{definition}
    A pair $(s,t)$ of positive integers is called \emph{regular} if there is no integer $r$ such that $0\leqslant r \leqslant t - 2$ and $2(r+1) = s - t + 1$.
\end{definition}

Any algebra $A$ with $\dim \mm/\mm^2 = h$ is isomorphic to an algebra of the form $\KK[x_1, x_2,\ldots, x_h]/I$ for some ideal $I\subset \KK[x_1,x_2, \ldots, x_h]$. Let $I, J\subset \KK[x_1,x_2, \ldots, x_h]$ be two ideals. For brevity, we call $I$ and $J$ isomorphic if the algebras $\KK[x_1,x_2, \ldots, x_h]/I$ and $\KK[x_1,x_2, \ldots, x_h]/J$ are isomorphic.

The main result of \cite{ElVa} is the following theorem.

\begin{theorem}\cite[Theorems 2.8 and 3.5]{ElVa}\label{classif}
    Let $A = \KK[x_1, x_2, \ldots, x_h]/I$ be an almost stretched algebra of type $(s,t)$. Suppose that $s\geqslant 2t - 1$. If $(s,t)$ is regular, then $I$ is isomorphic to one of the following ideals:
    $$I_{0,1}, I_{1,1}, \ldots, I_{t-1,1},$$
    where
\begin{equation*}
    I_{r,1} = (\underset{1\leq i < j \leq h,
    (i,j) \ne (1,2)}{x_ix_j}, \underset{3\leq j \leq h}{x_j^2 - x_1^s}, x_2^2 - x_1^{r+1}x_2 - x_1^{s-t+1}, x_1^tx_2).
\end{equation*}
    Moreover, if $s\geqslant 2t$, then these $t$ ideals are pairwise not isomorphic, so we have precisely $t$ isomorphism classes of Gorenstein algebras.
\end{theorem}

In the following sections, we will see that most Gorenstein algebras corresponding to non-degenerate projective hypersurfaces in $\PP^n$ of degree at least $n-3$ admitting an induced additive action are almost stretched. This is why we need the results listed above. The only exception is the algebras with the Hilbert-Samuel sequence $(1,2,3,1,1,\ldots, 1)$. For this sequence, we prove the following proposition.

\begin{proposition}\label{onetwothree}
    There are no Gorenstein local algebras with the Hilbert-Samuel sequence $(1,2,3,1,1,\ldots, 1)$.
\end{proposition}

\begin{proof}
    Let $A$ be an algebra with this Hilbert-Samuel sequence. Take any elements $x, y$ forming a basis of $\mm/\mm^2$ such that $\Soc A = \langle x^{n-3}\rangle$. Then $A$ has the following basis: $$1, x, y, x^2, xy, y^2, x^3, x^4, \ldots, x^{n-3}.$$ Since $\mm^3/\mm^4 = \langle x^3\rangle$, we have the equality $x^2y = x^3f(x)$ for some polynomial $f$. Rewriting this relation as $x^2(y-xf(x))$ and replacing $y$ with $y - xf(x)$ (denoting $y-xf(x)$ again by~$y$ for simplicity), we obtain the relation $x^2y = 0$.

Next, we have $xy^2 = g(x)$ for some polynomial $g(x)$ divisible by $x^3$. Multiplying this equality by $x$, we obtain $0 = x\cdot xy^2 = x^2y^2 = xg(x)$ which implies $$xy^2 = g(x) = ax^{n-3}, a\in\KK.$$

Similarly, $y^3 = h(x)$ and $xy^3 = xh(x)$. But $xy^3 = xy^2\cdot y = ax^{n-3}y = 0$, so $$y^3 = h(x) = bx^{n-3}, b\in\KK.$$

We see that multiplication by $x$ and $y$ defines linear maps from $\langle xy, y^2, x^{n-4}\rangle$ to $\langle x^{n-3}\rangle$. For dimension reasons, the kernels of both maps are at least $2$-dimensional, so they have a non-trivial intersection. This intersection lies inside $\Soc A$, so $A$ is not Gorenstein. The proposition is proved.

\end{proof}

In the conclusion of this section, we discuss the following question. Recall that, by the Hassett-Tschinkel correspondence, projective hypersurfaces in $\PP^n$ with an induced additive actions are in bijection with $H$-pairs $(A, U)$, where $A$ is a local algebra of dimension $n+1$ with the maximal ideal $\mm$ and $U$ is a hyperplane in $\mm$ generating $A$ as a unital algebra. In addition, the hypersurface is non-degenerate if and only if $A$ is Gorenstein and $U$ is complementary to $\Soc A$. It is conjectured that in this case the hypersurface depends only on the algebra $A$ and not on the subspace $U$. Equivalentely, this means that the group $\Aut(A)$ of all automorphisms of the unital algebra $A$ acts transitively on the set of subspaces in $\mm$ complementary to $\Soc A$. Although we are not able to prove this conjecture in general, we do it for all algebras which appear below. Now, we are going to describe the technique that we use for this.

Suppose that $G$ is a unipotent algebraic subgroup of $\Aut(A)$. Denote by $C$ the set of all $(n-1)$-dimensional subspaces of $\mm$ complementary to $\Soc A$. The set $C$ can be interpreted as an affine subspace $\AA^{n-1}$ in $\PP(\mm^*)$. Let $1, S_1, S_2, \ldots, S_n$ be a basis of $A$ such that $\Soc A = \langle S_n\rangle$. We can write any $z\in\mm$ as $z = z_1S_1 + z_2S_2 + \ldots + z_nS_n$ and consider $z_1, z_2, \ldots, z_n$ as a basis of $\mm^*$.

It is convenient to view the action of $G$ on $\mm^*$ in the following way. Let $f\in G$ be any element. Apply $f$ to the expression $z = z_1S_1 + z_2S_2 + \ldots + z_nS_n$:
$$f(z) = z_1f(S_1) + z_2f(S_2) + \ldots + z_nf(S_n)$$
and rewrite each $f(S_i)$ as a linear combination of $S_1, S_2, \ldots, S_n$. After that, the result of the action of $f$ on $z_i$ is the coefficient of $z_i$ in the expression we obtained.

Denote by $H$ the stabilizer of $z_n$ in $G$. The main statement that we will use is the following.

\begin{lemma}\label{orbit}
    Suppose that $\dim G - \dim H = n - 1$. Then $G$ acts transitively on $C$ and hence the corresponding hypersurface does not depend on the choice of the hyperplane in $\mm$ complementary to $\Soc A$.
\end{lemma}

\begin{proof}
    The equality $\dim G - \dim H = n - 1$ means that the $G$-orbit of the line $\KK z_{n}$ (viewed as a point in $C \cong \AA^{n-1}$) is $(n-1)$-dimensional. Since any orbit of a unipotent group acting on an affine variety is closed by \cite[Section 1.3]{PV}, it follows that $G$ acts transitively on $C$, i.e. any subspace of $\mm$ complementary to $\Soc A$ can be sent to any other such subspace by an automorphism from $G\subseteq \Aut(A)$. It follows that the hypersurface corresponding to $A$ does not depend on the choice of $U$.
\end{proof}

An example of application of Lemma \ref{orbit} can be found in \cite[Section 6]{ABeZa} for the algebra $A = \KK[x]/(x^{n+1})$. More examples are given in this paper in the following sections.

\section{Hypersurfaces in $\PP^n$ of degree $n-1$}
\label{codim_one}

In this section, we classify non-degenerate hypersurfaces in $\PP^n$ ($n \geq 3$) of degree $n - 1$ with an induced additive action. The main result is the following theorem. 

\begin{theorem}
    For any $n\geqslant 3$, there is exactly one, up to isomorphism, non-degenerate hypersurface of degree $n-1$ in $\PP^n$ admitting an induced additive action. The corresponding Gorenstein local algebra is $$\KK[x,y]/(xy, y^2 - x^{n-1}).$$
\end{theorem}

\begin{proof}

Let $A$ be a local Gorenstein algebra corresponding to such a hypersurface. Then we have $\dim A = n + 1$, $\length(A) = n$, so the Hilbert-Samuel sequence $(r_0, r_1, \ldots, r_{n-1})$ is $(1,2,1,1,\ldots, 1)$. Indeed, $\dim A - \length(A) = 1$ if and only if all numbers $r_i$ except one are equal to $1$, while the remaining one equals $2$. If $r_1 = 1$, then $r_i = 1$ for all $i$, so we have $r_1 = 2$. This means that $A$ is isomorphic to a quotient $\KK[x,y]/I$, where $I$ is an ideal, $x$ and~$y$ form a basis of $\mm/\mm^2$, and $\Soc A = \mm^{n-1}$. By Lemma \ref{degrees}, there is an element $z\in\mm\setminus \mm^2$ such that $\mm^{n-1} 
= \langle z^{n-1}\rangle$. We can assume that $z = x$, so $A$ has the following basis: $1, x, y, x^2, x^3, \ldots, x^{n-1}$. 

Since $xy\in\mm^2$, we have $xy = x^2f(x)$, where $f$ is some polynomial. Rewriting this equality as $x(y - xf(x)) = 0$, we can replace $y$ with $\tilde{y} = y - xf(x)$ obtaining the relation $x\tilde{y} = 0$. For simplicity of notation, we will write just $y$ instead of $\tilde{y}$.

Next, we have $y^2 = x^kg(x)$, where $g(x)$ is a polynomial such that $g(0)\ne 0$ and $k\geq 2$. Multiplying by $x$, we obtain $xy^2 = x^{k+1}g(x)$. But $xy^2 = xy\cdot y = 0$, which implies $k = n-1$. So, $y^2 = ax^{n-1}$ for some $a\in\KK$. The coefficient $a$ cannot be zero, since otherwise $y\in\Soc A$ and $A$ is not Gorenstein. So, $a\ne 0$, and, multiplying $x$ by a suitable coefficient, we arrive at the relation $y^2 = x^{n-1}$. So, $A$ is isomorphic to $\KK[x,y]/(xy, y^2 - x^{n-1})$, and it is easy to check that this algebra is indeed Gorenstein.

This proves that there is only one, up to isomorphism, Gorenstein algebra of dimension $n + 1$ and length $n$. It remains to show that the corresponding non-degenerate hypersurface does not depend on the choice of the hyperplane $U\in\mm$ complementary to $\Soc A = \langle x^{n-1}\rangle$. To do this, we are going to apply Lemma \ref{orbit}. Consider the following maps $f$ from $\KK[x,y]$ to~$\KK[x,y]$:
\begin{equation*} f(x) = x + a_2x^2 + a_3x^3 + \ldots + a_{n-1}x^{n-1}, \quad f(y) = y + bx^{n-1},
\end{equation*}
where $a_2, a_3, \ldots, a_{n-1}, b \in \KK$. It is easy to check that $f(I) \subseteq I$, so any such $f$ defines an automorphism of the algebra $A$. The composition of any two maps of this form is also a map of this form, so all such maps form an $(n-1)$-dimensional unipotent linear algebraic subgroup $G$ of $\Aut(A)$. We can write any $z\in\mm$ as $z = z_1x + z_2x^2 + \ldots + z_{n-1}x^{n-1} + wy$. The stabilizer of the form $z_{n-1}$ in the group~$G$ is trivial. Indeed, let us apply $f$ to $z = z_1x + z_2x^2 + \ldots + z_{n-1}x^{n-1} + wy$. If at least one of the coefficients $a_2, a_3, \ldots, a_{n-1}$ does not vanish, then $z_{n-1}$ is sent to the expression $z_{n-1} + (n - k)a_kz_{n-k+1} + \ldots \ne z_{n-1}$, where $k$ is the minimal index such that $a_k \ne 0$. If $a_2 = a_3 = \ldots = a_{n-1} = 0$, then $z_{n-1}$ is sent to $z_{n-1} + wb$, so $b$ also vanishes. So, we are done by Lemma \ref{orbit}.
\end{proof}

Let us write down the equation of the hypersurface corresponding to the algebra $\KK[x,y]/(xy, y^2 - x^{n-1})$. By \cite[Theorem 2.14]{AZa}, it has the form
$$z_0^{n-1}\pi\Big(\ln\big(1 + \frac{z}{z_0}\big)\Big) = \pi\Big(\sum_{k=1}^{n-1}\frac{(-1)^{k-1}}{k}z_0^{n-k-1}(z_1x + z_2x^2 + \ldots + z_{n-1}x^{n-1} + w_1y\big)^k\Big),$$
where $\pi\colon \mm \to \mm/U_0$ is the projection. This gives the equation
$$\sum_{k=1}^{n}\frac{(-1)^{k-1}}{k}z_0^{n-k}\sum_{j_1 + \ldots + j_k = n}c_jz_{j_1}\ldots z_{j_k} - \frac{1}{2}z_0^{n-3}w_1^2 = 0.$$

For example, for $n = 4$ and $n = 6$ we have the equations, respectively:
$$z_0^2z_3 - z_0z_1z_2 - \frac{1}{2}z_0w_1^2 + \frac{1}{3}z_1^3 = 0,$$
$$z_0^4z_5-z_0^3z_1z_4 - z_0^3z_2z_3 - \frac{1}{2}z_0^2w_1^2 + z_0^2z_1^2z_3 + z_0^2z_1z_2^2 - z_0z_1^3z_2 + \frac{1}{5}z_1^5 = 0.$$

It follows from \cite[Proposition 3]{ABeZa} that this hypersurface is normal only for $n = 3, 4$.

\section{Hypersurfaces in $\PP^n$ of degree $n-2$}
\label{codim_two}

In this section, we consider the case of hypersurfaces in $\PP^n$ ($n\geqslant 4$) of degree $n-2$. The main result of this section is the following theorem.

\begin{theorem}
    For any $n\geqslant 6$, there are, up to isomorphism, exactly three non-degenerate hypersurfaces of degree $n-2$ in $\PP^n$ admitting an induced additive actions. The corresponding Gorenstein local algebras are 
\begin{equation*}
\KK[x,y]/(y^2 - xy - x^{n-3}, x^2y),\quad
\KK[x,y]/(y^2 - x^{n-3}, x^2y) \text{ and}
\end{equation*}
\begin{equation*}
\KK[x,y]/(xy, xz, yz, y^2 - x^{n-2}, z^2 - x^{n-2}).
\end{equation*}

In $\PP^5$, there are two such hypersurfaces, and the corresponding Gorenstein local algebras are
\begin{equation*}
\KK[x,y]/(xy, x^3 - y^3)\text{ and } \KK[x,y]/(x^3, y^2).
\end{equation*}

In $\PP^4$, there is one such hypersurface isomorphic to the non-degenerate quadric, and the corresponding Gorenstein local algebra is
$$\KK[x,y,z]/(yz, xz, z^2 - x^2, y^2 - x^2).$$
\end{theorem}

\begin{proof}

Since $\dim A - \length(A) = 2$, it follows from Lemma \ref{alm_str} that the Hilbert-Samuel sequence $(r_0, r_1, \ldots, r_{n-2})$ of $A$ is of one of the following forms: $(1,2,2,1,\ldots, 1)$ or $(1,3,1,1,\ldots, 1)$. We consider the two cases separately.

\textit{Case 1:} $(r_0, r_1, \ldots, r_{n-2}) = (1, 2, 2, 1, \ldots, 1).$ This means that $A$ is an almost stretched algebra of type $(s,t) = (n-2,2)$. In this case, the inequality $s = n-2 \geqslant 2t = 4$ holds for $n\geqslant 6$. Moreover, the pair $(n-2,2)$ is regular unless $n = 5$. So, Theorem \ref{classif} can be applied for $n\geqslant 6$. Writing $x,y$ instead of $x_1, x_2$, we have
$$I_{0,1} = (y^2 - xy - x^{n-3}, x^2y), \quad I_{1,1} = (y^2 - x^2y-x^{n-3}, x^2y) = (y^2 - x^{n-3}, x^2y),$$
so $A$ is isomorphic to $A_0^2 = \KK[x,y]/(y^2-xy-x^{n-3}, x^2y)$ or $A_1^2 = \KK[x,y]/(y^2 - x^{n-3}, x^2y)$. Now, we are going to show that the hypersurfaces corresponding to $A_0^2$ and $A_1^2$ do not depend on the choice of the hyperplane in $\mm$ complementary to $\mm^{n-2} = \langle x^{n-2}\rangle$.

The algebra $A_0^2$ has the basis $1, x, y, x^2, y^2, x^3, \ldots x^{n-3}$. The following relations hold:
$$xy^2 = x(xy + x^{n-3}) = x^2y + x^{n-2} = x^{n-2},$$
$$y^3 = y(xy + x^{n-3}) = xy^2 + x^{n-3}y = x^{n-2},$$
and it follows that $x^iy^j = 0$ for all pairs of non-negative integers $(i,j)$ such that $i + j \geqslant 4$ unless $j = 0$. Consider the following maps $f\colon \KK[x,y] \to \KK[x,y]$:
\begin{equation*}
f(x) = x + a_2x^2 + a_3x^3 + \ldots + a_{n-2}x^{n-2}, \quad
f(y) = y + b_1y^2 + b_{n-3}x^{n-3} + b_{n-2}x^{n-2},
\end{equation*}
where $a_2, \ldots, a_{n-2}, b_1, b_{n-3}, b_{n-2}\in\KK$. It can be checked by a direct computation that the relation $x^2y = 0$ is always preserved by $f$, while the relation $y^2 = xy + x^{n-3}$ is preserved if and only if $b_1 - b_{n-3} = (n-3)a_2$. So, if $b_1 - b_{n-3} = (n-3)a_2$, then each map of this form induced an automorphism of the algebra $A_0^2$. It is also easy to check that the composition of two maps of this form is also a map of this form, so all such maps form an $(n-1)$-dimensional unipotent linear algebraic subgroup of $\Aut(A_0^2)$. Next, we apply $f$ to $$z = z_1x + z_2x^2 + \ldots + z_{n-2}x^{n-2} + w_1y + w_2y^2.$$
If $f$ stabilizes the form $z_{n-2}$, then $a_2 = a_3 = \ldots = a_{n-2} = 0$ by the same argument as in the proof of Theorem \ref{codim_one}. So, $z_{n-2}$ is sent to
$$z_{n-2} + w_1b_{n-2} + 2w_2b_1,$$
and $b_1 = b_{n-2} = 0$, which implies $b_{n-3} = b_1 - (n-3)a_2 = 0$. So, the stabilizer of $z_{n-2}$ in $G$ is trivial and this concludes the proof by Lemma \ref{orbit}.
\bigskip

The algebra $A_1^2$ has the same basis as $A_0^2$, i.e. $1, x, y, x^2, xy, x^3, \ldots, x^{n-2}$, and satisfies the following relations:
$$xy^2 = x\cdot x^{n-3} = x^{n-2},\quad y^3 = y\cdot x^{n-3} = 0.$$

Consider the maps of the form
\begin{equation*} f(x) = x + a_2x^2 + a_3x^3 + \ldots + a_{n-2}x^{n-2}, \quad
f(y) = y + b_1xy + b_{n-3}x^{n-3} + b_{n-2}x^{n-2}.
\end{equation*}

Now, $f$ preserves $I_{1,1}$ if and only if $(n-3)a_2 = 2b_1$, giving, as in the case of $A_0^2$, an $(n-1)$-dimensional unipotent subgroup of $\Aut(A_1^2)$. Again, we apply $f$ to
$$z = z_1x + z_2x^2 + \ldots + z_{n-2}x^{n-2} + w_1y + w_2xy.$$
If $f$ stabilizes $z_{n-2}$, then by $a_2 = a_3 = \ldots = a_{n-2} = 0$ by the same arguments as before implying also $b_1 = 0$. So, $z_{n-2}$ is sent to 
$$z_{n-2} + b_{n-2}w_1 + b_{n-3}w_2,$$
and $b_{n-2} = b_{n-3} = 0$. So, the stabilizer of $z_{n-2}$ in $G$ is trivial and we are done as before.

\bigskip

The only cases not covered by Theorem \ref{classif} are $n = 4$ and $n = 5$. If $n = 4$, then the Hilbert-Samuel sequence is $(1,2,2)$ and the algebra $A$ is not Gorenstein. If $n = 5$, then $\dim A = 6$ and the Hilbert-Samuel sequence of $A$ is $(1,2,2,1)$. All finite-dimensional local algebras of dimension up to $6$ are classified in \cite[Table 1]{AZa}. According to Table 1 from this work, there are two Gorenstein local algebra with the Hilbert-Samuel sequence $(1,2,2,1)$, namely $B_1 = \KK[x,y]/(xy, x^3 - y^3)$ and $B_2 = \KK[x,y]/(x^3, y^2)$.

For $B_1$, consider the automorphisms of the form
\begin{equation*}
f(x) = x + a_2x^2 + a_3x^3, \quad
f(y) = y + b_2y^2 + b_3y^3.
\end{equation*}

It is easy to see that any such $f$ is indeed an automorphism of $B_1$ and the stabilizer of the form $z_2$ in the unipotent group $G$ consisting of all such automorphisms is trivial. Since $\dim G = 4 = 6 - 2$, we are done by Lemma \ref{orbit} as before.

For $B_2 = \KK[x,y]/(x^3, y^2)$, we first apply the following change of variables: $x = x_1 - \frac{y_1}{3}$, $y = y_1$. Under this change, the algebra $B_2$ turns to $\KK[x_1,y_1]/(x_1^3 - x_1^2y_1, y_1^2)$. For simplicity, we replace $x_1, y_1$ with $x,y$ and work with the algebra $\widetilde{B}_2 = \KK[x,y]/(x^3-x^2y, y^2)$. For this algebra, consider the following automorphisms:
\begin{equation*}
f(x) = x + a_2x^2 + a_3x^3, \quad
f(y) = y + b_2xy + b_3x^3.
\end{equation*}

It is easy to see that any such $f$ is indeed an automorphism of $\widetilde{B}_2$ and all such maps form a unipotent subgroup of $\Aut(\widetilde{B}_2)$. If we apply $f$ to $$z_1x + z_2x^2 + z_3x^3 + z_4y + z_5xy,$$ then $z_3$ goes to
$$z_3 + 2a_2z_2 + a_3z_1 + b_3z_4 + b_2z_5 + a_2z_5$$
implying that the stabilizer of $z_3$ in $G$ is trivial. This again finishes the proof by Lemma \ref{orbit}.

\bigskip

\textit{Case 2:} $(r_0, r_1, \ldots, r_{n-2}) = (1, 3, 1, 1, \ldots, 1).$
Algebras with this Hilbert-Samuel sequence are almost stretched with $h = 3$, $t = 1$ and $s = n - 2$. The inequality $s = n - 2 \geqslant 2t = 2$ holds for all $n\geqslant 4$, so we do not have to consider exceptional cases here. By Theorem~\ref{classif}, there is only one Gorenstein local algebra for each $n\geqslant 4$, namely $A = \KK[x,y,z]/I_{0,1}$, where
$$I_{0,1} = (xz, yz, z^2 - x^{n-2}, y^2 - xy - x^{n-2}, xy) = (xy, xz, yz, y^2 - x^{n-2}, z^2 - x^{n-2}).$$

Consider the $(n-1)$-dimensional unipotent subgroup of $\Aut(A)$ consisting of the following automorphisms:
\begin{equation*}
f(x) = x + a_2x^2 + a_3x^3 + \ldots + a_{n-2}x^{n-2}, \quad f(y) = y + bx^{n-2}, \quad f(z) = z + cx^{n-2}.
\end{equation*}

Similarly to all previous cases, the stabilizer of the form $z_n$ under the action on $$z_1x + z_2x^2 + \ldots + z_{n-2}x^{n-2} + wy + uz$$ in the group $G$ is trivial, so we are done.

\end{proof}

Let us take a closer look, for example, at the hypersurface corresponding to the algebra $\KK[x,y]/(y^2 - xy - x^{n-3}, x^2y)$. By \cite[Theorem 2.14]{AZa}, it has the form
$$z_0^{n-2}\pi\Big(\ln\big(1 + \frac{z}{z_0}\big)\Big) = \pi\Big(\sum_{k=1}^{n-2}\frac{(-1)^{k-1}}{k}z_0^{n-k-2}(z_1x + z_2x^2 + \ldots + z_{n-2}x^{n-2} + w_1y + w_2xy\big)^k\Big),$$
where $\pi\colon \mm \to \mm/U_0$ is the projection. For example, the equations for $n = 5$ and $n = 6$ are the following, respectively:
$$z_0^2z_3 - z_0(z_1z_2 + w_1w_2) + \frac{1}{3}(z_1^3 + w_1^3 + 3z_1w_1^2) = 0,$$
$$z_0^3z_4 - \frac{1}{2}z_0^2(z_2^2 + 2z_1z_3 + 2w_1w_2) + \frac{1}{3}z_0(w_1^3 + 2z_1w_1^2 + 3z_1^2z_2) - \frac{1}{4}z_1^4 = 0.$$

By \cite[Proposition 3]{ABeZa}, these two hypersurfaces are normal. One can check that the hypersurfaces corresponding to the algebra $\KK[x,y]/(x^2y, y^2 - xy - x^{n-3})$ are not normal for~$n \geq 7$.

\section{Hypersurfaces in $\PP^n$ of degree $n - 3$}

By Lemma \ref{alm_str}, Gorenstein local algebras corresponding to non-degenerate hypersurfaces of degree $n-3$ in $\PP^n$ ($n \geq 5$) admitting an induced additive action can have the following Hilbert-Samuel sequences:
\begin{equation*}
(1, 4, 1, 1, 1, \ldots, 1),\quad
(1, 3, 2, 1, 1, \ldots, 1),\quad
(1, 2, 3, 1, 1, \ldots, 1),\quad
(1, 2, 2, 2, 1, \ldots, 1).
\end{equation*}
The third case is not possible by Proposition \ref{onetwothree}. We consider each of the remaining three cases.

\subsection{Hilbert-Samuel sequence $(1,4,1,\ldots, 1)$}

Local algebras with the Hilbert-Samuel sequence $(r_0, r_1, \ldots, r_{n-3}) = (1,4,1,\ldots, 1)$ are almost stretched of type $(s, t) = (n-3,1)$ with $h = 4$. The inequality $s = n - 3 \geqslant 2t = 2$ holds for any $n\geqslant 5$, so Theorem \ref{classif} can be applied. According to Theorem \ref{classif}, the only one Gorenstein local algebra with this Hilbert-Samuel sequence is $A = \KK[x,y,z,t]/I_{0,1}$, where
$$I_{0,1} = (xy, xz, xt, yz, yt, zt, z^2 - x^{n-3}, t^2 - x^{n-3}, y^2 - x^{n-3}).$$
In order to show that the corresponding hypersurface does not depend on the choice of the hyperplane in $\mm$, we consider the following automorphisms of $A$:
\begin{equation*}
\begin{split}
& f(x) = x + a_2x^2 + a_3x^3 + \ldots + a_{n-3}x^{n-3},\\
& f(y) = y + bx^{n-3}, \quad
f(z) = z + cx^{n-3}, \quad
f(t) = t + dx^{n-3},
\end{split}
\end{equation*}
where $a_2, \ldots, a_{n-3}, b,c,d \in \KK$. Using Lemma \ref{orbit}, we conclude that the corresponding hypersurface does not depend on the choice of $U\subset \mm$.

\subsection{Hilbert-Samuel sequence $(1,3,2,1,\ldots, 1)$}

Local algebras with the Hilbert-Samuel sequence $(r_0, r_1, \ldots, r_{n-3}) = (1, 3, 2, 1, \ldots, 1)$ are almost stretched of type $(s,t) = (n-3, 2)$ and $h = 3$. The inequality $s = n-3 \geqslant 2t = 4$ holds for $n\geqslant 7$, so Theorem \ref{classif} applies for all $n \geqslant 7$. In this case, there are two, up to isomorphism, local Gorenstein algebras, namely $A_0^3 = \KK[x,y,z]/I_{0,1}$ and $A_1^3 = \KK[x,y,z]/I_{1,1}$, where
$$I_{0,1} = (xz, yz, z^2 - x^{n-3}, y^2 - xy - x^{n-4}, x^2y),$$
$$I_{1,1} = (xz, yz, z^2 - x^{n-3}, y^2 - x^2y - x^{n-4}, x^2y) = (xz, yz, z^2 - x^{n-3}, y^2 - x^{n-4}, x^2y).$$

The algebra $A_0^3$ is very similar to the algebra $A_0^2$ considered in Section \ref{codim_two}, Case 1. The difference is the new variable $z$ such that $zx = zy = 0$ and $z^2 = x^{n-3}$. To prove the independence of the corresponding projective hypersurface from the choice of the hyperplane in $\mm$, we consider the following automorphisms of $A_0^3$:
\begin{equation*}
\begin{split}
& f(x) = x + a_2x^2 + a_3x^3 + \ldots + a_{n-3}x^{n-3}, \\
& f(y) = y + b_1y^2 + b_{n-4}x^{n-4} + b_{n-3}x^{n-3},\quad
f(z) = z + cx^{n-3}.
\end{split}
\end{equation*}
where $a_2, \ldots, a_{n-2}, b_1, b_{n-4}, b_{n-3}\in\KK$ and $b_1 - b_{n-4} = (n-4)a_2$. The remaining part of the proof is almost the same as for the algebra $A_0^2$.

For the algebra $A_1^3$, the situation is similar. Indeed, the only difference of $A_1^3$ from $A_1^2$ is the new variable $z$ such that $zx = zy = 0$ and $z^2 = x^{n-3}$. To prove that the corresponding hypersurface does not depend on the choice of the hyperplane in $\mm$, we consider the following automorphisms of $A_0^3$:
\begin{equation*}
\begin{split}
& f(x) = x + a_2x^2 + a_3x^3 + \ldots + a_{n-3}x^{n-3}, \\
& f(y) = y + b_1xy + b_{n-4}x^{n-4} + b_{n-3}x^{n-3},\quad
f(z) = z + cx^{n-3},
\end{split}
\end{equation*}
where $(n-4)a_2 = 2b_1$,
and the remaining part of the proof is similar to the one for $A_1^2$.

\bigskip

It remains to consider the exceptional cases $n = 5$ and $n = 6$. The case $n = 5$ is not possible, since an algebra with the Hilbert-Samuel sequence $(1,3,2)$ cannot be Gorenstein. For $n = 6$, the Hilbert-Samuel sequence becomes $(1,3,2,1)$, and the dimension of any such algebra is equal to $7$. All local Gorenstein algebras with dimension up to $9$ are classified in~\cite{Casn}. In particular, it follows from \cite{Casn} that each local Gorenstein algebras with the Hilbert-Samuel sequence $(1,3,2,1)$ is isomorphic to one of the following algebras:
\begin{equation*}
\begin{split}
& B_1 = \KK[x,y,z]/(xy, xz, yz, y^3 - x^3, z^2 - x^3), \\
& B_2 = \KK[x,y,z]/(x^2y - x^3, y^2, xz, yz, z^2 - x^3),\\
& B_3 = \KK[x,y,z]/(x^2y, y^2 - x^2, xz, yz, z^2 - x^3).
\end{split}
\end{equation*}

The change of variables $y_1 = x - y$, $x_1 = x + y$, $z_1 = 2z$ sends the ideal $$(x^2y, y^2 - x^2, xz, yz, z^2 - x^3)$$ to the ideal $(x_1y_1, x_1z_1, y_1z_1, y_1^3 - x_1^3, z_1^2 - x_1^3)$, so the algebras $B_1$ and $B_3$ are isomorphic. At the same time, the algebras $B_1$ and $B_2$ are not isomorphic. Indeed, consider the maps $\varphi_i\colon \mm_i/\mm_i^2 \to \mm_i^2/\mm_i^3$ ($i = 1, 2$) sending $s\in\mm_i/\mm_i^2$ to $s^2$. The equation $\varphi(s) = 0$ defines a one-dimensional subspace in $\mm_1/\mm_1^2$ and a two-dimensional subspace in $\mm_2/\mm_2^2$, so the two algebras are not isomorphic.

To prove that the hypersurface corresponding to $B_1$ does not depend on the choice of the hyperplane complementary to $\Soc B_1$, consider the automorphisms of $B_1$ of the form
\begin{equation*}
f(x) = x + a_2x^2 + a_3x^3, \quad f(y) = y + b_2y^2 + b_3x^3,\quad
f(z) = z + cx^3,
\end{equation*}
while for $B_2$ consider the automorphisms of the form
\begin{equation*}
f(x) = x + a_2x^2 + a_3x^3, \quad
f(y) = y + b_2xy + b_3x^3, \quad
f(z) = z + cx^3.
\end{equation*}

The remaining part of the proof where we apply Lemma \ref{orbit} is the same as in all previous cases.

\subsection{Hilbert-Samuel sequence $(1,2,2,2,1,1,\ldots, 1)$}

Local algebras with the Hilbert-Samuel sequence $(r_0, r_1, \ldots, r_{n-3}) = (1, 2,2,2,1,1,\ldots, 1)$ are almost stretched of type $(s,t) = (n-3,3)$. The inequality $s = n-3 \geqslant 2t = 6$ holds for $n\geqslant 9$ and the pair $(s,t)$ is regular unless $s = 6$ (i.e., $n = 9$) in which case the equality $2(r + 1) = s - t + 1$ holds for $r = 1$. So, for $n \geqslant 10$, there are exactly three isomorphism classes of local Gorenstein algebras of type $(n-3,3)$, namely $B_i = \KK[x,y]/I_{i,1}$, $i = 0, 1, 2$, where
\begin{equation*}\label{3_ideals}
I_{0,1} = (y^2 - xy - x^{n-5}, x^3y),\quad
I_{1,1} = (y^2 - x^2y - x^{n-5}, x^3y),\quad
I_{2,1} = (y^2 - x^3y - x^{n-5}, x^3y).
\end{equation*}

For each of these algebras, we need to establish independence of the corresponding non-degenerate hypersurface from the choice of the hyperplane in $\mm$. Let us start with $B_0$. For this algebra, we have the following relations:
$$xy^2 = x(xy + x^{n-5}) = x^2y + x^{n-4}, \quad y^3 = y(xy + x^{n-5}) = xy^2 + x^{n-5}y = x^2y + x^{n-4},$$
$$x^2y^2 = x^2(xy + x^{n-5}) = x^{n-3}, \quad xy^3 = x^{n-3}, \quad y^4 = x^{n-3}.$$

Consider the following maps $\KK[x,y]\to\KK[x,y]$:
\begin{equation*}
\begin{split}
& f(x) = x + a_2x^2 + a_3x^3 + \ldots + a_{n-3}x^{n-3} + cxy + dx^2y,\\
& f(y) = y + b_1xy + b_2x^2y + b_{n-5}x^{n-5} + b_{n-4}x^{n-4} + b_{n-3}x^{n-3}.
\end{split}
\end{equation*}

The relation $x^3y = 0$ is preserved by any $f$ of this form. As for $y^2 = xy + x^{n-5}$, we need to do the following computations:
$$f(y^2) = y^2 + b_1^2x^2y^2 + 2b_1xy^2 + 2b_2x^2y^2 = y^2 + b_1^2x^{n-3} + 2b_1(x^2y + x^{n-4}) + 2b_2x^{n-3},$$
$$f(xy) = f(x)f(y) = xy + b_1x^2y + b_{n-5}x^{n-4} + b_{n-4}x^{n-3} + a_2x^2y + a_2b_{n-5}x^{n-3} + cxy^2 + cb_1x^2y^2 + dx^2y^2 = $$
$$= xy + (b_1 + a_2 + c)x^2y + (b_{n-5} + c)x^{n-4} + (b_{n-4} + a_2b_{n-5} + b_1c + d)x^{n-3},$$
$$f(x^{n-5}) = x^{n-5} + (n-5)a_2x^{n-4} + \big((n-5)a_3 + \binom{n-5}{2}a_2^2\big)x^{n-3}.$$

This means that $f$ induces an automorphism of $B_0$ if and only if the following relations hold:
\begin{equation}\label{rel_0}
\begin{split}
& 2b_1 = b_1 + a_2 + c,\quad 2b_1 = b_{n-5} + c + (n-5)a_2,\\
& b_1^2 + 2b_2 = b_{n-4} + a_2b_{n-5} + b_1c + d + (n-5)a_3 + \binom{n-5}{2}a_2^2.
\end{split}
\end{equation}
The set of all $f$ satisfying these equations forms an $n$-dimensional unipotent subgroup of the group $\Aut(B_0)$.

Now, consider the action of $f$ on $z = z_1x + z_2x^2 + \ldots + z_{n-3}x^{n-3} + w_1y + w_2xy + w_3x^2y$. Let us compute the stabilizer in $G$ of the form $z_{n-3}$. It is easy to see that for any $k\geq 3$ the coefficient of $x^{n-3}$ in the expansion of $(x + a_2x^2 + a_3x^3 + \ldots + a_{n-3}x^{n-3} + cxy + dx^2y)^k$ does not contain $c$, $d$, so $a_2 = a_3 = \ldots = a_{n-5} = a_{n-3} = 0$ as in all similar examples considered previously and $2a_{n-4} + c^2 = 0$. Clearly, $b_{n-3} = 0$, while the coefficient of $x^{n-3}$ in the expansion of
$$w_2f(x)f(y) + w_3f(x)^2f(y) = w_2(x + a_{n-4}x^{n-4} + cxy + dx^2y)(y + b_1xy+ b_2x^2y + b_{n-5}x^{n-5} + b_{n-4}x^{n-4}) + $$
$$ + w_3(x^2 + 2cx^2y)(y + b_1xy+ b_2x^2y + b_{n-5}x^{n-5} + b_{n-4}x^{n-4})$$
is equal to
$$w_2(b_{n-4} + cb_1 + d) + w_3(b_{n-5} + 2c).$$

Together with (\ref{rel_0}), we arrive at the following equations:
$$2b_1 = b_1 + c, \quad
        2b_1 = b_{n-5} + c,\quad
        b_1^2 + 2b_2 = b_{n-4} + b_1c + d,\quad
        b_{n-4} + b_1c + d = 0,\quad
        b_{n-5} + 2c = 0.$$

The first two equations give us $b_1 = b_{n-5} = c$. From this and from the fifth equation, it follows that $b_1 = b_{n-5} = c = 0$, and the remaining equations give us $b_2 = 0$. So, the system is equivalent to the equation $b_{n-4} + d = 0$, so the stabilizer of $z_{n-3}$ in $G$ is one-dimensional. Since $\dim G = n$, it follows that the orbit of $\KK z_{n-3}$ has the desired dimension $n - 1$, and we are done by Lemma \ref{orbit} as before.
\bigskip

For the algebra $B_{1}$, we have the following relations:
$$xy^2 = x^{n-4}, \quad y^3 = x^2y^2 = x^{n-3}, \quad xy^3 = y^4 = 0,$$
while for $B_2$
$$xy^2 = x^{n-4},\quad x^2y^2 = x^{n-3}, y^3 = xy^3 = 0.$$

Consider the following maps $\KK[x,y]\to\KK[x,y]$:
\begin{equation*}
\begin{split}
& f(x) = x + a_2x^2 + a_3x^3 + \ldots + a_{n-3}x^{n-3},\\
& f(y) = y + b_1xy + b_2x^2y + b_{n-5}x^{n-5} + b_{n-4}x^{n-4} + b_{n-3}x^{n-3}.
\end{split}
\end{equation*}
The relation $x^3y = 0$ is again preserved by any $f$ of this form. As for the second relation, we have
$$f(y^2) = y^2 + b_1^2x^{n-3} + 2b_1x^{n-4} + 2b_2x^{n-3},$$
$$f(x^{n-5}) = x^{n-5} + (n-5)a_2x^{n-4} + ((n-5)a_3 + \binom{n-5}{2}a_2^2)x^{n-3},$$
$$f(x^2y) = x^2y + b_{n-5}x^{n-3}.$$

So, $f$ is an automorphism of $B_1$ if and only if $$2b_1 = (n-5)a_2, \quad b_1^2 + 2b_2 = (n-5)a_3 + \binom{n-5}{2}a_2^2 + b_{n-5}.$$

Similarly, $f$ is an automorphism of $B_2$ if and only if
$$2b_1 = (n-5)a_2,\quad b_1^2 + 2b_2 = (n-5)a_{3} + \binom{n-5}{2}a_2^2.$$

So, the unipotent group $G$ generated by all automorphisms of this form has dimension $n-1$ for both $B_1$ and $B_2$. As for $B_0$, consider the action of $f$ on $$z = z_1x + z_2x^2 + \ldots + z_{n-3}x^{n-3} + w_1y + w_2xy + w_3x^2y.$$
If $f$ stabilizes $z_{n-3}$, then, as before, $b_{n-3} = 0$ and $a_2 = a_3 = \ldots = a_{n-3} = 0$, which implies $b_1 = 0$. Similarly, $b_{n-5} = 0$ implying $b_{2} = 0$. Finally, one shows that $b_{n-4} = 0$, so the stabilizer of $z_n$ in $G$ is trivial and we are done by Lemma \ref{orbit}.
\bigskip

If $n = 9$, then the pair $(s,t) = (6,3)$ is not regular. Almost stretched Gorenstein algebras of type $(s,t)$ such that the pair $(s,t)$ is not regular and $s\geqslant 2t$ are also classified in \cite[Theorems 2.8 and 3.5]{ElVa}. It turns out that the number of isomorphism classes of such algebras is infinite. The case of the particular Hilbert-Samuel sequence $(1,2,2,2,1,1,1)$ is considered separately (Example 2 in Section 5). The corresponding local Gorenstein algebras are the two algebras $B_0 = \KK[x,y]/(y^2 - xy - x^4)$ and $B_2 = \KK[x,y]/(y^2 - x^4, x^3y)$ and the following infinite family of non-isomorphic algebras
$$B_1(c) = \KK[x,y]/(y^2 - x^2y - cx^4, x^3y), \quad c\in\KK^*.$$

For these algebras, the independence of the corresponding hypersurface from the choice of the hyperplane in the maximal ideal can be checked in the same way as for the algebra~$B_1$ for $n\geqslant 10$ which we did above.

The case of the Hilbert-Samuel sequence $(1,2,2,2,1,1)$, $n = 8$, is considered in \cite[Section~5, Example 4]{ElVa}. The corresponding Gorenstein algebras are again $\KK[x,y]/I_{i,1}$ ($i = 0,1,2$). However, $I_{1,1}\cong I_{2,1}$ if $n = 8$, so only there are only two isomorphism classes of algebras in this case.

Finally, if $n = 7$ and the Hilbert-Samuel sequence is $(1,2,2,2,1)$, then one can conclude from \cite[Section 5, Remark 5]{ElVa} that there are three isomorphism classes of Gorenstein algebras with this sequence, namely
$$C_1 = \KK[x,y]/(y^2 - x^2, x^3y),\quad C_2 = \KK[x,y]/(y^2 - x^3, x^4 - x^3y, x^5), \quad C_3 = \KK[x,y]/(y^2, x^4 - x^3y).$$

Note that $C_1$ is the algebra $\KK[x,y]/I_{2,1}$, so it remains to prove independence of the corresponding hypersurface from the choice of the hyperplane in $\mm$ for the algebras $C_2$ and~$C_3$. For the algebra $C_2$, we have the following relations:
$$xy^2 = y^3 = x^4,\quad x^2y^2 = xy^3 = y^4 = 0.$$

Consider the following maps $\KK[x,y]\to\KK[x,y]$:
\begin{equation*}
f(x) = x + a_2x^2 + a_3x^3 + a_4x^4,\quad f(y) = y + b_2xy + b_3x^2y + cx^3 + b_4x^4.
\end{equation*}

Since $f(x^3) = x^3 + 3a_2x^4$, $f(y^2) = y^2 + 2b_2x^4 + 2cx^4$, we see that $f$ is an automorphism of~$C_2$ if and only if $2b_2 + 2c = 3a_2$ and all maps $f$ of this kind form a $6$-dimensional unipotent subgroup $G\subset \Aut(C_2)$. Consider the action of $f$ on $$z = z_1x + z_2x^2 + z_3x^3 + z_4x^4 + w_1y + w_2xy + w_3x^2y.$$

As before, if $f$ stabilizes $z_4$, then $a_2 = a_3 = a_4 = 0$ and $b_4 = 0$. So, we have
$$f(z) = z_1x + z_2x^2 + z_3x^3 + z_4x^4 + (w_1 + w_2x + w_3x^2)(y + b_2xy + b_3x^2y + cx^3)$$
and $z_4$ is sent to
$$z_4 + w_3b_2 + w_2c + w_2b_3,$$
which means $b_3 + c = 0$ and $b_2 = 0$. The relation $2b_2 + 2c = 3a_2$ implies $c = b_3 = 0$, and the stabilizer of $z_4$ in $G$ is trivial, which is exactly what we need.

For $C_3$, consider the following maps:
\begin{equation*}
f(x) = x + a_2x^2 + a_3x^3 + a_4x^4,\quad
f(y) = y + b_2xy + b_3x^2y + b_4x^4.
\end{equation*}

The relations $y^2 = 0$ and $x^4 = x^3y$ are preserved by any $f$ of this kind, so the set of all such maps $f$ forms a $6$-dimensional unipotent subgroup of $\Aut(C_3)$. Again, consider the action of $f$ on
$$z = z_1x + z_2x^2 + z_3x^3 + z_4x^4 + w_1y + w_2xy + w_3x^2y.$$
Suppose that $f$ stabilizes $z_4$. Then, as before, $a_2 = a_3 = a_4 = 0$ and $b_4 = 0$, and we have
$f\colon z_n \mapsto z_n + w_2b_3 + w_3b_2$. So, $b_2 = b_3 = 0$, and we are done.

To conclude, we arrive at the following theorem giving a complete classification of non-degenerate projective hypersurfaces of degree $n-3$ in $\PP^n$ admitting an induced additive action.

\begin{theorem}\label{codim_three}
    The number of isomorphism classes of non-degenerate projective hypersurfaces of degree $n - 3$ in $\PP^n$ admitting an induced additive action and the corresponding Gorenstein algebras are given in Table 2 below. The algebras $A_5^9(c)$, $c\in\KK^*$, are pairwise non-isomorphic for different $c$.
\end{theorem}

\renewcommand{\arraystretch}{1.5}
\hspace{-1.2cm}
\begin{tabular}{ |c|c|l|  }
\hline
$n$ & \makecell{Number \\of hypersurfaces} & \multicolumn{1}{|c|}{The corresponding Gorenstein algebras}\\
\hline
$5$ & $1$ & $A_1^5 = \KK[x,y,z,t]/(xy, xz, xt, yz, yt, zt, z^2 - x^{2}, t^2 - x^{2}, y^2 - x^{2})$\\
\hline
\multirow{3}{*}{$6$} & \multirow{3}{*}{$3$} & $A_1^6 = \KK[x,y,z,t]/(xy, xz, xt, yz, yt, zt, z^2 - x^{3}, t^2 - x^{3}, y^2 - x^{3})$\\
& &$B_1 = \KK[x,y,z]/(xy, xz, yz, y^3 - x^3, z^2 - x^3)$\\
& & $B_2 = \KK[x,y,z]/(x^2y - x^3, y^2, xz, yz, z^2 - x^3)$\\
\hline
\multirow{6}{*}{$7$} & \multirow{6}{*}{$6$} & $A_1^7 = \KK[x,y,z,t]/(xy, xz, xt, yz, yt, zt, z^2 - x^{4}, t^2 - x^{4}, y^2 - x^{4})$\\
& &$A_2^7 = \KK[x,y,z]/(xz, yz, z^2 - x^{4}, y^2 - xy - x^{3}, x^2y)$\\
& & $A_3^7 = \KK[x,y,z]/(xz, yz, z^2 - x^{4}, y^2 - x^{3}, x^2y)$\\
& & $A_6^7 = C_1 =  \KK[x,y]/(y^2 - x^{2}, x^3y)$\\
& & $C_2 = \KK[x,y]/(y^2 - x^3, x^4 - x^3y, x^5)$\\
& & $C_3 = \KK[x,y]/(y^2, x^4 - x^3y)$\\
\hline
\multirow{5}{*}{$8$} & \multirow{5}{*}{$5$} & $A_1^8 = \KK[x,y,z,t]/(xy, xz, xt, yz, yt, zt, z^2 - x^{5}, t^2 - x^{5}, y^2 - x^{5})$\\
& &$A_2^8 = \KK[x,y,z]/(xz, yz, z^2 - x^{5}, y^2 - xy - x^{4}, x^2y)$\\
& & $A_3^8 = \KK[x,y,z]/(xz, yz, z^2 - x^{5}, y^2 - x^{4}, x^2y)$\\
& & $A_4^8 = \KK[x,y]/(y^2 - xy - x^{3}, x^3y)$\\
& & $A_5^8 \cong A_6^8 = \KK[x,y]/(y^2 - x^2y - x^{3}, x^3y)$\\
\hline
\multirow{6}{*}{$9$} & \multirow{6}{*}{$\infty$} & $A_1^9 = \KK[x,y,z,t]/(xy, xz, xt, yz, yt, zt, z^2 - x^{6}, t^2 - x^{6}, y^2 - x^{6})$\\
& &$A_2^9 = \KK[x,y,z]/(xz, yz, z^2 - x^{6}, y^2 - xy - x^{5}, x^2y)$\\
& & $A_3^9 = \KK[x,y,z]/(xz, yz, z^2 - x^{6}, y^2 - x^{5}, x^2y)$\\
& & $A_4^9 = \KK[x,y]/(y^2 - xy - x^{4}, x^3y)$\\
& & $A_5^9(c) = \KK[x,y]/(y^2 - x^2y - cx^{4}, x^3y)$, $c\in\KK^*$\\
& & $A_6^9 = \KK[x.y]/(y^2 - x^3y - x^{4}, x^3y)$\\
\hline
\multirow{6}{*}{$\geq 10$} & \multirow{6}{*}{$6$} & $A_1^n = \KK[x,y,z,t]/(xy, xz, xt, yz, yt, zt, z^2 - x^{n-3}, t^2 - x^{n-3}, y^2 - x^{n-3})$\\
& &$A_2^n = \KK[x,y,z]/(xz, yz, z^2 - x^{n-3}, y^2 - xy - x^{n-4}, x^2y)$\\
& & $A_3^n = \KK[x,y,z]/(xz, yz, z^2 - x^{n-3}, y^2 - x^{n-4}, x^2y)$\\
& & $A_4^n = \KK[x,y]/(y^2 - xy - x^{n-5}, x^3y)$\\
& & $A_5^n = \KK[x,y]/(y^2 - x^2y - x^{n-5}, x^3y)$\\
& & $A_6^n = \KK[x.y]/(y^2 - x^3y - x^{n-5}, x^3y)$\\
\hline
\end{tabular}

\vspace{0.05cm}

\renewcommand{\arraystretch}{1}

\begin{center}
Table 2: Non-degenerate projective hypersurfaces of degree $n-3$ in $\PP^n$\\
admitting an induced additive action
\end{center}

\vspace{0.3cm}

Let us take a closer look at the infinite family of non-degenerate hypersurfaces $X_c \subseteq\PP^9$ corresponding to the algebras $\KK[x,y]/(y^2 - x^2y - cx^4, x^3y)$. By \cite[Theorem 2.14]{AZa}, their equations are
$$z_0^{6}\pi\Big(\ln\big(1 + \frac{z}{z_0}\big)\Big) = \pi\Big(\sum_{k=1}^{6}\frac{(-1)^{k-1}}{k}z_0^{8-k}(z_1x + \ldots + z_{6}x^{6} + w_1y + w_2xy + w_3x^2y\big)^k\Big) =$$
$$ = z_0^5f_1 + z_0^4f_2 + z_0^3f_3 + z_0^2f_4 + z_0f_5 + f_6$$
where $\pi\colon\mm \to \mm/U$ is the projection and each $f_i$ is a homogeneous polynomial in $z_1, \ldots, z_6, w_1, w_2, w_3$ of degree $i$. We have the relations $xy^2 = cx^5$ and $y^3 = x^2y^2 = cx^6$, so the polynomials $f_6$ and $f_5$ do not depend on $w_1, w_2, w_3$ and are equal, respectively, to $-\frac{1}{6}z_1^6$ and $z_1^4z_2$. It follows from \cite[Proposition 3]{ABeZa} that all these hypersurfaces are not normal.


\section*{Statements and Declarations}

The article was prepared within the framework of the project “International academic cooperation” HSE University.

\section*{Conflict of Interest}

The authors have no relevant financial or non-financial interests to disclose.

\section*{Data Availability}

Data sharing not applicable to this article as no datasets were generated or analysed during the current study.

\end{document}